\documentclass{amsart}
\usepackage{amssymb, enumerate}
\usepackage{amsrefs}
\usepackage{graphicx}

\theoremstyle{plain}
\newtheorem{theorem}{Theorem}
\newtheorem{lemma}[theorem]{Lemma}
\newtheorem*{thmsep}{Theorem \ref{sep}}

\newtheorem{prop}[theorem]{Proposition}
\newtheorem{corollary}[theorem]{Corollary}

\theoremstyle{definition}

\theoremstyle{remark}
\newtheorem{remark}[theorem]{Remark}

\newcommand{\lp}{\left(}

\newcommand{\rp}{\right)}
\newcommand{\e}{\epsilon}
\newcommand{\R}{\mathbb{R}}

\newcommand{\Sp}{\mathbb{S}}

\newcommand{\Z}{\mathbb{Z}}

\newcommand{\twoheaddownarrow}{\mathrel{\rotatebox{90}{$\twoheadleftarrow$}}}

\DeclareMathOperator{\dist}{\textup{\text{dist}}}
\DeclareMathOperator{\diam}{\textup{\text{diam}}}
\DeclareMathOperator{\id}{\textup{\text{id}}}

\DeclareMathOperator{\Vol}{\textup{\text{Vol}}}

\DeclareMathOperator{\Area}{\textup{\text{Area}}}

\DeclareMathOperator{\innt}{\textup{\text{int}}}

\DeclareMathOperator{\conv}{\textup{conv}}
\DeclareMathOperator{\Ner}{\textup{Ner}}
\DeclareMathOperator{\inrad}{\textup{in-rad}}
\DeclareMathOperator{\seprad}{\textup{sep-rad}}

\numberwithin{equation}{section}
\numberwithin{theorem}{section}

\begin{document}

\title[Lower bounds for codimension-1 measure]{Lower bounds for codimension-1 measure in metric manifolds}
\author{Kyle Kinneberg}
\address{Department of Mathematics, Rice University, 6100 Main Street, Houston TX 77005}
\email{kyle.kinneberg@rice.edu}

\date{\today}

\begin{abstract}
We establish Euclidean-type lower bounds for the codimension-1 Hausdorff measure of sets that separate points in doubling and linearly locally contractible metric manifolds. This gives a quantitative topological isoperimetric inequality in the setting of metric manifolds, in the sense that lower bounds for the codimension-1 measure of a set depend not on some notion of filling or volume but rather on in-radii of complementary components. As a consequence, we show that balls in a closed, connected, doubling, and linearly locally contractible metric $n$-manifold $(M,d)$ with radius $0<r\leq \diam(M)$ have $n$-dimensional Hausdorff measure at least $c \cdot r^n$, where $c>0$ depends only on $n$ and on the doubling and linear local contractibility constants.
\end{abstract}

\maketitle

\section{Introduction}

According to the traditional Euclidean isoperimetric inequality, there is a dimensional constant $c_n >0$ such that
$$\Area(\partial E) \geq c_n \Vol(E)^{\frac{n-1}{n}}$$ 
for all closed sets $E \subset \R^n$, where ``$\Vol$" denotes $n$-dimensional Hausdorff measure and ``$\Area$" denotes $(n-1)$-dimensional Minkowski content. This inequality, along with the related filling inequalities, have played a central role in the development of geometric analysis and have led to similar results in other geometric settings, including Riemannian manifolds, sub-Riemannian manifolds, discrete graphs, finitely-generated groups, and certain classes of metric measure spaces. Depending on the context, one must interpret $E$, $\partial E$, $\Vol$, and $\Area$ in the correct way; and depending on the geometry involved, the resulting isoperimetric inequalities might look quite different from their Euclidean counterparts. Still, the common theme in them all is to bound $\Area(\partial E)$ from below by some function of $\Vol(E)$.

This theme, however, can break down in highly non-smooth metric settings. Consider, for instance, a metric on $\Sp^2$ that contains subsets at various small scales that have wildly different Hausdorff dimensions but are bounded by finite-length curves. Such metrics could be constructed as Hausdorff limits of certain polyhedral complexes in $\R^3$, defined iteratively by cubical sub-division and replacement rules, much like the ``snow-sphere" constructions of D. Meyer \cite{Mey}. By altering the subdivision rules in different regions that are very far from each other in relative distance, it is possible to make the resulting metric spheres highly non-homogeneous for Hausdorff dimension. Even if there was an appropriate notion of volume in such a space, it would be difficult to find a meaningful relationship between the volumes of sets and the areas, or lengths, of their boundaries.

At the same time, these types of non-smooth metric spaces can arise as geometrically controlled deformations of spaces that do support traditional isoperimetric inequalities. For example, the metric spheres described above are quasisymmetrically equivalent to the Euclidean sphere (this equivalence can even be realized by a continuous family of quasiconformal deformations of $\R^3$, as shown in \cite{Mey}). This observation prompts the natural question: Can one establish isoperimetric-type inequalities for a broad, quasisymmetrically-invariant class of metric spaces that still gives meaningful geometric information?

Our primary purpose in this paper is to prove the following result, which provides a possible answer to this question for the class of doubling and linearly locally contractible closed metric manifolds. It gives a quantitative relationship between the size of a set, measured by its in-radius and the in-radius of its complement, and the area of its boundary, measured by codimension-1 Hausdorff measure.

\begin{theorem} \label{isothm}
Let $(M,d)$ be a closed, connected, metric manifold of dimension $n \geq 1$ that is $D$-doubling and $L$-linearly locally contractible. For $E \subset M$ Borel,
$$\mathcal{H}_{n-1}(\partial E) \geq c \cdot \min(\inrad(E), \inrad(M\backslash E))^{n-1},$$
where $c>0$ depends only on $n$, $D$, and $L$.
\end{theorem}

Our investigations fall naturally into the setting of quantitative topology, and as such, the methods we use are more topological than geometric. Consequently, the inequalities we obtain mimic what one would expect in Euclidean geometry, but they do not account for the important effects that other geometries can have on isoperimetric relationships. In many ways, this is a necessary loss if we want to study such a broad class of spaces.

\subsection{Definitions and results}

Let $(Z,d)$ be a metric space. Following standard notation, $B(z,r)$ will denote the open ball centered at $z \in Z$ with radius $r>0$.

For $\e >0$, a collection of points $z_1,\ldots,z_m \in Z$ is said to be $\e$-separated if $d(z_i,z_j) \geq \e$ for all $i,j$ distinct. We say that $Z$ is $D$-doubling if any ball $B(z,r)$ contains at most $D$ points that are $r/2$-separated. This is quantitatively equivalent to the condition that every ball can be covered by at most $D'$ balls of half the radius, but we prefer to use the previous formulation as it is easier to work with. A metric space is said to be doubling if it is $D$-doubling for some $D \geq 1$. This property basically functions as a finite dimensionality condition for $Z$.

We say that $Z$ is $L$-linearly locally contractible if any ball $B(z,r)$ with radius $0<r \leq \diam(Z)/L$ can be contracted inside $B(z,Lr)$ to a point. More precisely, there is a continuous map $h \colon [0,1] \times B(z,r) \rightarrow B(z,Lr)$ with $h(0,\cdot)$ the identity inclusion and $h(1,\cdot)$ a constant map. Naturally, $Z$ is said to be linearly locally contractible if this condition holds for some $L \geq 1$. This property prevents thin necks, long fingers, and cusps in $Z$, and generally speaking, it guarantees that topologically large sets have topologically large boundaries (cf. Corollary \ref{homcor}). In a vague sense, then, such spaces have a weak topological ``isoperimetric inequality," and Theorem \ref{isothm} can be viewed as a quantitative extension of this fact.

Both doubling and linear local contractibility have appeared frequently in quasiconformal geometry, especially in the context of quasisymmetric parameterization problems \cite{BK, HeinICM, HS, Sem96, SemNon, Wild}. Indeed, both properties are preserved under quasisymmetric homeomorphisms. They have also appeared in the study of Gromov hyperbolic groups: the boundary of a hyperbolic group equipped with a visual metric is always doubling, and if the boundary is a topological sphere, then it is necessarily linearly locally contractible (cf. Theorem 3.3 in \cite{KleinICM}). Perhaps more familiarly, it is easy to see that a closed Riemannian manifold is doubling and linearly locally contractible, though the associated constants will depend strongly on the geometry of the manifold.

For $E \subset Z$ Borel, we use $\mathcal{H}_k(E)$ to denote the $k$-dimensional Hausdorff measure of $E$. The in-radius of $E$ is defined to be
$$\inrad(E) = \sup \{r \geq 0 : B(x,r) \subset E \text{ for some } x \in Z\}.$$
If $E$ has no interior, note that $\inrad(E)=0$. It will also be convenient for us to define, for $S \subset Z$ closed, the separation radius of $S$ to be
$$\seprad(S) = \sup \left\{ \min(\inrad(U),\inrad(V)) : 
\begin{array}{l}
U,V \text{ are distinct connected} \\
\text{components of } Z \backslash S
\end{array} \right\}.$$
If $Z$ is path connected and $\seprad(S) > r$, then clearly there are points $x,y \in Z$ with $\dist(S,\{x,y\}) > r$ that are separated by $S$, in the sense that any path from $x$ to $y$ necessarily intersects $S$. For convenience, we set $\seprad(S) = 0$ whenever the complement $Z \backslash S$ is connected.

In most of our investigations below, we will focus on metric spaces that have the structure of a topological manifold. Namely, a metric manifold $(M,d)$ is a topological manifold $M$, endowed with a metric $d$, that induces the manifold topology. This should not be confused with the notion of a metrizable manifold, as we concern ourselves with the fixed metric $d$ rather than the fact that $M$ admits some metric. Note that if $M$ is connected, then it is necessarily path connected. 

It is not difficult to see that there is a soft relationship between separation radius and codimension-1 Hausdorff measure for closed sets in metric manifolds. Namely, if $M$ is an $n$-dimensional manifold and $S \subset M$ has $\seprad(S) > 0$, then there are points $x,y \in M \backslash S$ that are separated by $S$. It is a standard fact that the topological dimension of $S$ must therefore be at least $n-1$ (cf. Theorem IV.4 in \cite{HW} and the subsequent corollary), and this implies that $\mathcal{H}_{n-1}(S) >0$ (cf. Theorem VII.2 in \cite{HW}). Thus, positivity of separation radius implies positivity of $(n-1)$-dimensional Hausdorff measure.

We will prove Theorem \ref{isothm} by showing that this relationship is quantitative when $M$ is a doubling and linearly locally contractible closed metric manifold. In fact, the precise bounds we obtain coincide with the expected bounds in Euclidean space.

\begin{theorem} \label{sep}
Let $(M,d)$ be a closed, connected, metric manifold of dimension $n \geq 1$ that is $D$-doubling and $L$-linearly locally contractible. If a closed set $S \subset M$ separates two points $x,y$ in $M$ with $\dist(S, \{x,y\}) \geq r$, then 
$$\mathcal{H}_{n-1}(S) \geq c \cdot r^{n-1},$$ 
where $c>0$ depends only on $n$, $D$, and $L$. In particular, for any closed set $S \subset M$, we have $\mathcal{H}_{n-1}(S) \geq c \cdot \seprad(S)^{n-1}$.
\end{theorem}

\begin{remark} \label{isormk}
Observe that Theorem \ref{isothm} follows directly from Theorem \ref{sep}. Indeed, for $E \subset M$ Borel, the closed set $S = \partial E$ satisfies the hypothesis in Theorem \ref{sep} with $r = \min(\inrad(E), \inrad(M\backslash E))$. The remainder of this paper will therefore focus on Theorem \ref{sep}.
\end{remark}

Before moving on, let us explain why the imposed conditions on $M$ are necessary. First, we note that it really is necessary to define the separation radius using two distinct complementary components of $S$, and consequently that both $\inrad(E)$ and $\inrad(M\backslash E)$ must appear in the isoperimetric relationship. Indeed, the boundary of a tiny ball in $\Sp^n$ has tiny $(n-1)$-dimensional measure, but it has a complementary component with very large in-radius. To understand why we require $M$ to be a closed manifold, consider a solid dumbbell: a compact 3-manifold with boundary formed by connecting two solid balls of unit radius with a very thin solid cylinder, say with unit length and radius $\e>0$. This space is doubling and linearly locally contractible with constants independent of $\e$. However, a disk that arises as a slice of the cylinder has separation radius comparable to 1, while its 2-dimensional Hausdorff measure is comparable to $\e^2$. For the necessity of linear local contractibility (or a similar condition), consider the boundary of the solid dumbbell. This is topologically $\Sp^2$ and is doubling with constant independent of $\e$, but loops that surround the cylinder have separation radius comparable to 1 and length comparable to $\e$. Of course, the linear local contractibility constant here is comparable to $1/\e$.

Another immediate consequence of Theorem \ref{sep} is that metric balls have volume at least as large as comparison balls in Euclidean space, up to a quantitative constant. An analogous statement for Riemannian manifolds was obtained by R.~E.~Greene and P.~Petersen in \cite{GP}. There, the authors use local contractibility estimates to obtain lower bounds on the codimension-1 volume of geodesic spheres, and this translates into lower bounds on the volume of balls. 

\begin{corollary} \label{vol}
Let $(M,d)$ be a closed, connected, metric manifold of dimension $n \geq 1$ that is $D$-doubling and $L$-linearly locally contractible. Then for each $x \in M$ and $0<r \leq \diam(M)$, we have
$$\mathcal{H}_n(B(x,r)) \geq c \cdot r^n,$$
where $c >0$ depends only on $n$, $D$, and $L$.
\end{corollary}

\begin{proof}
Fix $x \in M$ and $0<r \leq \diam(M)$. Consider the $1$-Lipschitz function $z \mapsto d(x,z)$ on the ball $B(x,r)$, and let $S_t = \{z : d(x,z)=t \}$ denote its level sets. By Eilenberg's inequality (cf. Theorem 13.3.1 in \cite{BZ}), we have
$$\int_{[0,r)}^* \mathcal{H}_{n-1}(S_t) dt \leq C \cdot \mathcal{H}_n(B(x,r)),$$
where $\int^*$ denotes the upper integral (to avoid issues of measurability) and $C$ depends only on $n$. As $r \leq \diam(M)$ there is a point $y \in M$ with $d(x,y) \geq r/2$, so for any $r/8 \leq t \leq r/4$, the points $x$ and $y$ are separated by $S_t$ and have $\dist(S_t,\{x,y\}) \geq r/8$. Theorem \ref{sep} then guarantees that $\mathcal{H}_{n-1}(S_t) \geq c\cdot r^{n-1}$ for $r/8 \leq t \leq r/4$, and this gives the desired bound.
\end{proof}

We should remark that Corollary \ref{vol} is not entirely new. Indeed, a deep theorem of S. Semmes, Theorem 1.29(a) in \cite{Sem96}, guarantees the following much stronger property. For each $x \in M$ and $0<r \leq \diam(M)$ as above, there is a surjective map $f \colon M \rightarrow \Sp^n$ that is $C/r$-Lipschitz and is constant outside of $B(x,r/2)$. Again, the constant $C$ depends only on $n$, $D$, and $L$. As $n$-dimensional Hausdorff measure can increase by at most a factor of $\lambda^n$ under a $\lambda$-Lipschitz map, Corollary \ref{vol} is a trivial consequence of the existence of such a map $f$. 

At the same time, constructing this map $f$ requires a lot of work, and the lower volume bound statements for balls are much weaker. Moreover, statements about codimension-1 volume bounds do not seem to follow immediately from Semmes's results. We should remark, however, that our proof of Theorem \ref{sep} uses many of the ideas developed in \cite{Sem96}. Our presentation of them is, as one should expect, somewhat simpler because our goals are less ambitious.

Before setting off, let us give a broad outline of the proof of Theorem \ref{sep}. By scaling the metric, we may assume that $r=1$, so our goal is to show that $\mathcal{H}_{n-1}(S)$ has a uniform lower bound. The main idea, which comes from \cite{Sem96}, is to approximate $M$ by a simplicial complex $\mathcal{S}$ that arises as the nerve of an appropriately chosen open cover of $M$. The complex $\mathcal{S}$ approximates $M$ in a partially quantitative way, in that there exists a uniformly Lipschitz map $f \colon M \rightarrow \mathcal{S}$ and a continuous map $g \colon \mathcal{S} \rightarrow M$ for which $g \circ f$ is homotopic to the identity on $M$, through a homotopy that moves points by distance at most $1/4$.

Now, suppose that $\mathcal{H}_{n-1}(S)$ was very small. Then its image $f(S)$ in $\mathcal{S}$ would still have very small $(n-1)$-dimensional measure. Using the simplicial structure of $\mathcal{S}$, an argument from the heart of the Federer--Fleming Deformation Theorem allows us to project $f(S)$ into the $(n-2)$-dimensional skeleton of $\mathcal{S}$. Applying the map $g$ to the projected image takes us back into $M$, giving a closed set $S'$ that is homotopic to $S$, again through a homotopy of $M$ that moves points by distance at most $1/4$. One can show that $S'$ must still separate the points $x$ and $y$ in $M$, and in principle this should contradict the fact that $S'$ is the image of an $(n-2)$-dimensional object. However, we have no control on the modulus of continuity for $g$, so we cannot conclude that $S'$ has topological dimension at most $n-2$. Instead, we will use some homology arguments to obtain the desired contradiction.

The structure of this paper is essentially the reverse of the outline just given. In the next section we develop the homological tools needed for the ending argument. In Section \ref{secdef} we prove the statement about projecting small sets in simplicial complexes into appropriate-dimensional skeleta. For the most part, this is a simplified version of the Federer--Fleming Deformation Theorem, though we do it for sets rather than for currents. Section \ref{secapprox} is devoted to approximating a metric space by a simplicial complex, and it is here that linear local contractibility plays an essential role. In Section \ref{secend} we prove Theorem \ref{sep}, which is not difficult given the content of the previous sections.

\subsection*{Acknowledgements}
It is a pleasure to thank Mario Bonk for many conversations related to this project, which arose in discussions about the author's dissertation project. The author also thanks Stephen Semmes for pointing out the reference \cite{GP} and the anonymous referee for helpful comments.

\section{Some necessary topology}

In this section, we establish Lemma \ref{hom}, which contains the topological tools needed to prove Theorem \ref{sep}. No metrics appear here, but we make substantial use of homology and cohomology groups of topological spaces. As we do not want to deal with issues of orientation, all of these groups will have coefficients in $\Z_2$. Generally, we work with singular homology $H_\ast(X)$ or its relative version $\tilde{H}_\ast(X)$, and we work with \v{C}ech cohomology $\check{H}^\ast(X)$ to avoid problems with potential pathologies of the topological space $X$. 

\begin{lemma} \label{hom}
Let $M$ be a closed connected manifold, and let $S \subset M$ be a closed subset. If $S$ separates two points $x,y \in M$, then $\check{H}^{n-1}(S)$ is non-trivial. Moreover, if $h \colon S \rightarrow M$ is a continuous map that is homotopic to the inclusion $\iota \colon S \hookrightarrow M$ through maps whose images are disjoint from $\{x,y\}$, then the induced homomorphism $h^\ast \colon \check{H}^{n-1}(h(S)) \rightarrow \check{H}^{n-1}(S)$ is non-trivial, and the closed set $h(S)$ separates $x$ and $y$ as well.
\end{lemma}

\begin{proof}
We first show that $\check{H}^{n-1}(S)$ is non-trivial. By duality (see p. 296 in \cite{Spa}) there is an isomorphism between $\check{H}^{n-1}(S)$ and the relative singular homology group $H_1(M,M\backslash S)$, again with coefficients in $\Z_2$. As $H_1(M,M\backslash S)$ and its reduced version $\tilde{H}_1(M,M\backslash S)$ are isomorphic, it suffices to show that $\tilde{H}_1(M,M\backslash S) \neq 0$. To this end, consider the long exact sequence of reduced relative homology groups
$$\cdots \rightarrow \tilde{H}_1(M,M\backslash S) \rightarrow \tilde{H}_0(M\backslash S) \rightarrow \tilde{H}_0(M) \rightarrow 0,$$
and note that $\tilde{H}_0(M) = 0$ because $M$ is connected. In particular, the map $\tilde{H}_1(M,M\backslash S) \rightarrow \tilde{H}_0(M\backslash S)$ is surjective. As $S$ separates $x$ and $y$ in $M$, the set $M \backslash S$ has at least two path components, which implies that $\tilde{H}_0(M\backslash S) \neq 0$. By surjectivity, we conclude that $\tilde{H}_1(M,M\backslash S) \neq 0$ as well.

We now turn to the second part of the lemma, where $h \colon S \rightarrow M$ is continuous and homotopic to the inclusion $\iota \colon S \hookrightarrow M$ via a homotopy whose trace set $T$ is disjoint from $\{x,y\}$. Note that $T \subset M$ is a closed subset and that there are natural inclusion maps
$$\iota_1 \colon h(S) \hookrightarrow T \hspace{0.3cm} \text{ and } \hspace{0.3cm} \iota_2 \colon S \hookrightarrow T.$$
For simplicity, let $h_2 \colon S \rightarrow T$ be given by $h_2 = \iota_1 \circ h$, so that $h_2$ is homotopic to the map $\iota_2$ in $T$.

All of these maps induce dual homomorphisms between the relevant \v{C}ech cohomology groups. Note that $h_2^\ast \colon \check{H}^{n-1}(T) \rightarrow \check{H}^{n-1}(S)$ factors as $h_2^\ast = h^\ast \circ \iota_1^\ast$, where $\iota_1^\ast \colon  \check{H}^{n-1}(T) \rightarrow \check{H}^{n-1}(h(S))$ and $h^\ast \colon  \check{H}^{n-1}(h(S)) \rightarrow \check{H}^{n-1}(S)$. Thus, to show that $h^\ast$ is non-trivial, it suffices to show that $h_2^\ast$ is non-trivial. Moreover, as $h_2$ is homotopic to $\iota_2$ in $T$, the maps $h_2^\ast$ and $\iota_2^\ast \colon \check{H}^{n-1}(T) \rightarrow {H}^{n-1}(S)$ are the same. Thus, we wish to show that $\iota_2^\ast$ is non-trivial.

To this end, consider the following diagram, where the vertical maps are obtained from what we established in the first part of the lemma.
$$\begin{array}{ccc}
\check{H}^{n-1}(T)  &   \stackrel{\iota_2^\ast}{\longrightarrow}   &  \check{H}^{n-1}(S) \\
\updownarrow         &                                                                          & \updownarrow \\
\tilde{H}_1(M,M\backslash T) &             &    \tilde{H}_1(M,M\backslash S) \\
\twoheaddownarrow   &                 &    \twoheaddownarrow \\
\tilde{H}_0(M\backslash T) &    \longrightarrow   &    \tilde{H}_0(M\backslash S)
\end{array}$$
The map at the bottom of the diagram is induced by the inclusion $M\backslash T \hookrightarrow M\backslash S$. Thus, the horizontal maps are both induced by inclusions, and as the vertical maps are natural with respect to inclusion, the diagram commutes. 

By assumption, $T$ is disjoint from $\{x,y\}$, so that $[x]$ and $[y]$ are elements of $\tilde{H}_0(M\backslash T)$. Abusing notation, we also use $[x]$ and $[y]$ to denote the images of these elements in $\tilde{H}_0(M\backslash S)$ under the bottom map in the diagram. Observe that $[x]-[y] \neq 0$ in both homology groups because $S$, and hence $T$, separates $x$ and $y$. 

The vertical maps in the diagram are surjective, so there is $\alpha \in \check{H}^{n-1}(T)$ that maps to $[x]-[y]$ in $\tilde{H}_0(M\backslash T)$. In particular $\alpha \neq 0$. Now note that $\iota_2^\ast(\alpha) \in \check{H}^{n-1}(S)$ maps to the non-zero element $[x]-[y]$ in $\tilde{H}_0(M\backslash S)$ because the diagram commutes. We can conclude that $\iota_2^\ast(\alpha) \neq 0$, so the map $\iota_2^\ast$ is non-trivial.

Finally, we show that the closed set $h(S)$ separates $x$ and $y$. To this end, observe that we have a similar diagram to the one used above, where $S$ is replaced by its image $h(S)$.
$$\begin{array}{ccc}
\check{H}^{n-1}(T)  &   \stackrel{\iota_1^\ast}{\longrightarrow}   &  \check{H}^{n-1}(h(S)) \\
\updownarrow         &                                                                          & \updownarrow \\
\tilde{H}_1(M,M\backslash T) &     \stackrel{j_\ast}{\longrightarrow}        &    \tilde{H}_1(M,M\backslash h(S)) \\
\twoheaddownarrow   &                 &    \\
\tilde{H}_0(M\backslash T) &       &   
\end{array}$$
Recall that $\iota_1 \colon h(S) \hookrightarrow T$ is simply the inclusion. Now we also use $j \colon M \backslash T \hookrightarrow M\backslash h(S)$ to denote the complementary inclusion, and $j_\ast$ is the induced homomorphism on reduced relative homology groups. Once again, this diagram commutes because all horizontal maps are obtained from inclusions.

Let $\gamma$ be a path in $M$ with endpoints $x$ and $y$, and let $[\gamma]$ denote the corresponding element in $\tilde{H}_1(M,M\backslash T)$. The vertical map $\tilde{H}_1(M,M\backslash T) \twoheadrightarrow \tilde{H}_0(M\backslash T)$ sends $[\gamma]$ to $[x]-[y]$, which is non-zero (recall that we are working with $\Z_2$ coefficients, so the orientation on $\gamma$ does not matter). Moreover, if $\alpha \in \check{H}^{n-1}(T)$ is the element that corresponds to $[\gamma]$ under the isomorphism between $\check{H}^{n-1}(T)$ and $\tilde{H}_1(M,M\backslash T)$, then the vertical map on the left sends $\alpha$ to $[x]-[y]$ in $\tilde{H}_0(M\backslash T)$. By what we established above, we know that $\iota_2^\ast(\alpha)$ is a non-zero element of $\check{H}^{n-1}(S)$.

We claim that $\iota_1^{\ast}(\alpha)$ is non-zero in $\check{H}^{n-1}(h(S))$ as well. Indeed, recall that $h_2 = \iota_1 \circ h$, so that $h_2^\ast = h^\ast \circ \iota_1^\ast$. However, we also have $h_2^\ast = \iota_2^\ast$ because $h_2$ is homotopic to $\iota_2$, so we can write $\iota_2^\ast = h^\ast \circ \iota_1^\ast$. Using that $\iota_2^\ast(\alpha) \neq 0$, it is clear that $\iota_1^\ast(\alpha) \neq 0$ also.

As the diagram above commutes, we know that $j_\ast([\gamma])$ is a non-zero element of $\tilde{H}_1(M,M\backslash h(S))$. Of course, $j_\ast([\gamma])$ is obtained simply by viewing the path $\gamma$ as an element in the group $\tilde{H}_1(M,M\backslash h(S))$. Hence, $j_\ast([\gamma]) \neq 0$ implies that $\gamma$ intersects the set $h(S)$, which is what we needed to show.
\end{proof}

At this point, it might seem as if the statement ``$h^\ast \colon \check{H}^{n-1}(h(S)) \rightarrow \check{H}^{n-1}(S)$ is non-trivial" is only a technical tool to prove the softer (and seemingly obvious) statement that $h(S)$ separates $x$ and $y$. In fact, it is precisely the non-triviality of this group homomorphism that we will need to use later.

Let us also note that it is here in Lemma \ref{hom} where we see the importance of $M$ being a closed manifold. For example, even Euclidean space fails to satisfy the conclusion of the lemma, as one can see by taking $S$ to be a hyperplane. Similarly, solid balls fail this lemma, and so do the solid dumbbells that we discussed in the previous section.

Before finishing this section, let us record a consequence of Lemma \ref{hom} for linearly locally contractible manifolds. We will not use this result in the sequel, but it justifies a remark from the introduction and serves as a weak version of the isoperimetric relationship.

\begin{corollary} \label{homcor}
Let $(M,d)$ be a closed, connected, metric manifold of dimension $n \geq 1$ that is $L$-linearly locally contractible. If $S \subset M$ is a closed subset, then
$$\diam(S) \geq \frac{1}{L}\seprad(S).$$
\end{corollary}

\begin{proof}
Clearly, we may assume that $\seprad(S) >0$. It suffices to show that $\diam(S) \geq r/L$ whenever $S$ separates two points $x,y \in M$ with $\dist(S, \{x,y\}) \geq r$. Fix such points $x$ and $y$, and suppose for a contradiction that $\diam(S) < r/L$.

Choose $z \in S$ so that $S \subset B(z,r/L)$. By linear local contractibility, $B(z,r/L)$ can be contracted within $B(z,r)$ to a point $z'$. In particular, the inclusion map $S \hookrightarrow M$ is homotopic to the constant map $h(S) = \{z'\}$ through maps whose images are disjoint from $\{x,y\}$. Lemma \ref{hom} then implies that $x$ and $y$ are separated by the point $z'$, which contradicts that $M$ is a closed, connected manifold. 
\end{proof}

\section{A projection lemma} \label{secdef}

Let $\Delta_n$ denote the $n$-dimensional simplex, which we realize geometrically as
$$\Delta_n = \conv(e_1,\ldots,e_{n+1}) \subset \R^{n+1},$$
the convex hull of the standard basis vectors in $\R^{n+1}$. We give $\Delta_n$ the metric coming from the Euclidean metric in $\R^{n+1}$.

\begin{lemma} \label{projlem}
For each $n \geq 1$, there is a constant $C_n \geq 1$ for which the following holds. If $E \subset \Delta_n$ is a closed set with $\mathcal{H}_k(E) < \infty$ and $k < n$, then there is a continuous map $p \colon \Delta_n \rightarrow \Delta_n$, fixing $\partial \Delta_n$ point-wise, for which $p(E) \subset \partial \Delta_n$ and $\mathcal{H}_k(p(E)) \leq C_n \mathcal{H}_k(E)$.
\end{lemma}

Philosophically, this is a simple case of the Federer--Fleming Deformation Theorem, proved in Section 5 of \cite{FF}. One must be careful, though, as our statement concerns fairly general sets $E$, so representing $E$ as a rectifiable current could lead to problems. As an alternative, we could cite Proposition 3.1 in \cite{DS00}, which proves a more general version of Lemma \ref{projlem}. For the purpose of completeness, though, we give a proof here using only what is really needed.

\begin{proof}
The proof essentially follows that of Lemma 3.22 in \cite{DS00}. Let $\frac{1}{2}\Delta_n \subset \Delta_n$ denote the $n$-dimensional simplex with the same barycenter as $\Delta_n$ but with side-length $\sqrt{2}/2$, i.e., half the side-length of $\Delta_n$. For $y \in \frac{1}{2}\Delta_n$, let $\theta_y \colon \Delta_n \backslash \{y\} \rightarrow \partial \Delta_n$ be the radial projection, so that for each $x \in \Delta_n \backslash \{y\}$, the segment between $y$ and $\theta_y(x)$ contains $x$. Notice that if $K \subset \Delta_n \backslash \{y\}$ is compact, then the restriction of $\theta_y$ to $K$ is Lipschitz with constant $\leq C \dist(y,K)^{-1}$, where $C \geq 1$ is universal. 

From this, it is not difficult to see that if $y \in \frac{1}{2}\Delta_n \backslash E$, then
\begin{equation} \label{projeq}
\mathcal{H}_k(\theta_y(E)) \leq C^k \int_E |x-y|^{-k} d\mathcal{H}_k(x).
\end{equation}
Integrating over $y \in \frac{1}{2}\Delta_n \backslash E$, and using the upper integral to avoid difficulties of measurability, we find that
$$\begin{aligned}
\int_{\frac{1}{2}\Delta_n \backslash E}^\ast \mathcal{H}_k(\theta_y(E))dy &\leq  C^k \int_{\frac{1}{2}\Delta_n} \int_E |x-y|^{-k} d\mathcal{H}_k(x) dy \\
&= C^k \int_E \lp \int_{\frac{1}{2}\Delta_n} |x-y|^{-k} dy \rp d\mathcal{H}_k(x),
\end{aligned}$$
where the equality comes from Fubini's theorem. As $k < n$, the singular integral $\int_{\frac{1}{2}\Delta_n} |x-y|^{-k} dy$ is uniformly bounded over $x \in E$, so we obtain
$$\int_{\frac{1}{2}\Delta_n \backslash E}^\ast \mathcal{H}_k(\theta_y(E))dy \leq C' \cdot \mathcal{H}_k(E),$$
where the constant $C'$ depends only on $n$. Using that $E$ has trivial $n$-dimensional measure, we see that there is $y_0 \in \frac{1}{2}\Delta_n \backslash E$ for which $\mathcal{H}_k(\theta_{y_0}(E)) \leq C_n \mathcal{H}_k(E)$.

Let $r = \dist(E,y_0) >0$, and let $p \colon \Delta_n \rightarrow \Delta_n$ be a continuous map that equals $\theta_{y_0}$ outside $B(y_0,r)$ and is the identity on $B(y_0,r/2)$. Then $p$ fixes $\partial \Delta_n$ point-wise, has $p(E) \subset \partial \Delta_n$, and satisfies $\mathcal{H}_k(p(E)) \leq C_n \mathcal{H}_k(E)$.
\end{proof}

\begin{remark}
It is not necessary for $E \subset \Delta_n$ to be closed for this argument to work. For example, it would suffice to assume that $E \cap K$ is closed for all compact sets $K \subset \Delta_n \backslash \partial \Delta_n$. We will use this formulation in the proof of the following proposition.
\end{remark}

If $\mathcal{S}$ is a finite simplicial complex, then we endow it with the intrinsic (path) metric which makes every $k$-dimensional sub-simplex $\sigma \subset \mathcal{S}$ isometric to $\Delta_k$. The dimension of $\mathcal{S}$ is defined to be the largest dimension of a simplex appearing in $\mathcal{S}$. As is standard, we use $\mathcal{S}^{(k)}$ to denote the $k$-dimensional skeleton of $\mathcal{S}$, which is the simplicial complex consisting of all simplices in $\mathcal{S}$ of dimension at most $k$. 

The following proposition is proved directly from the previous lemma, along with iteration. Here it is important that the maps $p$ in the lemma fix $\partial \Delta_n$.

\begin{prop} \label{proj}
For each $1\leq k \leq n$, there is a constant $c_{k,n} >0$ for which the following holds. Let $\mathcal{S}$ be a simplicial complex of dimension $n$, and let $E \subset \mathcal{S}$ be a closed set with $\mathcal{H}_k(E) \leq c_{k,n}$. Then there is a continuous map $p \colon \mathcal{S} \rightarrow \mathcal{S}$ with $p(\sigma) \subset \sigma$ for each sub-simplex $\sigma \subset \mathcal{S}$, and $p(E) \subset \mathcal{S}^{(k-1)}$.
\end{prop}

\begin{proof}
Fix $k \geq 1$. We prove the desired statement by induction on $n$, beginning with the case $n=k$. For the base case, let $c_{k,k} = \mathcal{H}_k(\Delta_k)/2$. If $\mathcal{S}$ is a $k$-dimensional simplicial complex and $E \subset \mathcal{S}$ is closed with $\mathcal{H}_k(E) \leq c_{k,k}$, then for any $k$-dimensional sub-simplex $\sigma \subset \mathcal{S}$, there is a point $x_{\sigma} \in \sigma \backslash E$. Using a radial projection away from $x_{\sigma}$ as we did at the end of the proof of the previous lemma, it is easy to build a continuous map on $\sigma$ that projects $E \cap \sigma$ into $\partial \sigma$ and fixes a small ball around $x_{\sigma}$. Piecing these projections together on the various $k$-dimensional simplices in $\mathcal{S} = \mathcal{S}^{(k)}$, and keeping $\mathcal{S}^{(k-1)}$ fixed, we obtain the desired map $p \colon \mathcal{S} \rightarrow \mathcal{S}$.

For the induction step, suppose we know that the statement holds for some value $n -1 \geq k$, with associated constant $c_{k,n-1} >0$. Let $\mathcal{S}$ be an $n$-dimensional simplicial complex with $E \subset \mathcal{S}$ closed and $\mathcal{H}_k(E) \leq c_{k,n-1}/C_n$, where $C_n$ is the constant from the previous lemma. Let $\sigma_1,\ldots,\sigma_m$ be the collection of $n$-dimensional sub-simplices in $\mathcal{S}$, and let $\innt(\sigma_i) = \sigma_i \backslash \partial \sigma_i$ denote the interior. Then $\mathcal{S}$ is the disjoint union
$$\mathcal{S} = \mathcal{S}^{(n-1)} \cup \innt(\sigma_1) \cup \cdots \cup \innt(\sigma_m).$$
For each $i$, let $p_i \colon \sigma_i \rightarrow \partial \sigma_i$ be the map obtained from Lemma \ref{projlem} applied to the set $E \cap \innt(\sigma_i)$. We should remark here that $E \cap \innt(\sigma_i)$ might not be closed, but its intersection with any compact subset of $\innt(\sigma_i)$ is. This is all that is needed for the proof of the lemma.

As $p_i$ fixes $\partial \sigma_i$, we can piece these together to obtain a map $\tilde{p} \colon \mathcal{S} \rightarrow \mathcal{S}$ that fixes $\mathcal{S}^{(n-1)}$ point-wise, for which $\tilde{p}(E) \subset \mathcal{S}^{(n-1)}$ and $\tilde{p}(\sigma) \subset \sigma$ for every sub-simplex $\sigma \subset \mathcal{S}$. We can then estimate
$$\begin{aligned}
\mathcal{H}_k(\tilde{p}(E)) &\leq \mathcal{H}_k(E \cap \mathcal{S}^{(n-1)}) + \sum_{i=1}^m \mathcal{H}_k(\tilde{p}(E \cap \innt(\sigma_i))) \\
&\leq \mathcal{H}_k(E \cap \mathcal{S}^{(n-1)}) + C_n \sum_{i=1}^m \mathcal{H}_k(E \cap \innt(\sigma_i)) \\
& \leq C_n \mathcal{H}_k(E)
\end{aligned}$$
Thus, $\tilde{E} = \tilde{p}(E) \subset \mathcal{S}^{(n-1)}$ is a closed set with $\mathcal{H}_k(\tilde{E}) \leq c_{k,n-1}$.

By the induction hypothesis, there is a continuous map $\hat{p} \colon \mathcal{S}^{(n-1)} \rightarrow \mathcal{S}^{(n-1)}$ with $\hat{p}(\sigma) \subset \sigma$ for each sub-simplex $\sigma \subset \mathcal{S}^{(n-1)}$ and $\hat{p}(\tilde{E}) \subset \mathcal{S}^{(k-1)}$. The first property ensures that we can continuously extend $\hat{p}$ to be defined on all of $\mathcal{S}$ in such a way that $\hat{p}(\sigma_i) \subset \sigma_i$ for each $n$-dimensional simplex $\sigma_i$ in $\mathcal{S}$. In particular, $\hat{p}(\sigma) \subset \sigma$ for each sub-simplex $\sigma \subset \mathcal{S}$. Now, we can define $p=\hat{p} \circ \tilde{p} \colon \mathcal{S} \rightarrow \mathcal{S}$, which has all of the desired properties. This completes the induction step with constant $c_{k,n} = c_{k,n-1}/C_n$.
\end{proof}

\section{Approximation by simplicial complexes} \label{secapprox}

We now turn our attention to compact metric spaces in general, using nerves of properly chosen coverings to produce simplicial approximations to the original space. This scheme is, in essence, topological, but as our ultimate goal is quantitative, we will need to make everything quantitative. We carry this out by closely following the ideas in Sections 4--5 of \cite{Sem96}.

Let $(M,d)$ be a compact metric space and let $\mathcal{U} = \{U_i\}_{i =1}^\ell$ be a finite open cover of it. We canonically associate to $\mathcal{U}$ a (geometric realization of a) finite simplicial complex, called the nerve of $\mathcal{U}$, as follows. Let $e_1,\ldots,e_\ell$ denote the standard basis vectors in $\R^\ell$. The nerve is defined to be
$$\Ner(\mathcal{U}) = \bigcup \{ \conv(e_{i_1},\ldots,e_{i_m}) : U_{i_1} \cap \cdots \cap U_{i_m} \neq \emptyset \} \subset \R^\ell,$$
where the union is taken over all collections $e_{i_1},\ldots,e_{i_m}$ for which the associated open sets $U_{i_1},\ldots,U_{i_m}$ have non-empty intersection. Note that each convex hull $\conv(e_{i_1},\ldots,e_{i_m})$ is an $(m-1)$-dimensional simplex. We endow $\Ner(\mathcal{U})$ with the intrinsic metric, as we did for simplicial complexes in the previous section. Observe, however, that the intrinsic metric is comparable to the restriction of the Euclidean metric in $\R^\ell$, with absolute constants. This follows from the fact that all simplices in $\Ner(\mathcal{U})$ are convex hulls of standard basis vectors. When the open cover $\mathcal{U}$ is fine enough, we think of $\Ner(\mathcal{U})$ as a good approximation to $M$. This should not be taken too literally, though.

Let $\mathcal{U} = \{U_i\}_{i =1}^\ell$ be a finite open cover of $M$, and let $\delta >0$ be a Lebesgue number of $\mathcal{U}$, so that each ball $B(x, \delta)$ in $M$ lies entirely in one of the open sets $U_i$. Also, let $N = \max \{ \sum_i \chi_{U_i}(x) : x \in M\}$ be the multiplicity of the cover. We can define a Lipschitz map $f\colon M \rightarrow \Ner(\mathcal{U})$ in the following way, where the Lipschitz constant depends only on $\delta$ and $N$. 

For each $1\leq i\leq \ell$, define
$$\phi_i(x) = \min \left\{ 1, \tfrac{2}{\delta} \dist(x,N_{\delta/2}(M\backslash U_i)) \right\},$$
where $N_{\delta/2}(M\backslash U_i)$ denotes the $\delta/2$-neighborhood of $M \backslash U_i$. Then $\phi_i \colon M \rightarrow [0,1]$ is $2/\delta$-Lipschitz and has support contained in $U_i$. Moreover, if $B(x,\delta) \subset U_i$, then $\phi_i(x) = 1$, so we know that $\phi(x) := \sum_i \phi_i(x) \geq 1$ for each $x \in M$. Now define $f_i \colon M \rightarrow [0,1]$ by $f_i(x) = \phi_i(x)/\phi(x)$ for each $1\leq i \leq \ell$. The collection $\{f_i\}_{i=1}^{\ell}$ forms a Lipschitz partition of unity for $M$, subordinate to the cover $\mathcal{U}$, in the sense that each $f_i$ is $(2N+1)/\delta$-Lipschitz with support contained in $U_i$, and $\sum_i f_i(x)=1$ for all $x \in M$.

Finally, we define $f \colon M \rightarrow \Ner(\mathcal{U})$ by
$$f(x) = \sum_{i=1}^\ell f_i(x)e_i,$$
where $e_1,\ldots,e_\ell$ are, again, the standard basis vectors in $\R^\ell$. It is straightforward to show, from the definitions, that $f(x)$ indeed lies in $\Ner(\mathcal{U})$ for each $x \in M$. Furthermore, we note that $f$ is Lipschitz with respect to the Euclidean metric on $\Ner(\mathcal{U})$, and so also is Lipschitz with respect to the intrinsic metric. Here, the Lipschitz constant depends only on $\delta$ and $N$.

In the setting of doubling and linearly locally contractible metric spaces, there is a stronger relationship between $M$ and $\Ner(\mathcal{U})$ than one expects in general, at least for properly chosen covers $\mathcal{U}$. Loosely, there is map $g \colon \Ner(\mathcal{U})\rightarrow M$ that acts as an inverse to the natural Lipschitz map $f$, in the sense that $g\circ f$ is homotopic to the identity on $M$. Proposition \ref{fg}, which is the main result of this section, makes this more precise. Before addressing the proposition, we must establish a couple of lemmas.

\begin{lemma} \label{cover}
Let $(M,d)$ be a compact metric space that is $D$-doubling. For each $\e>0$, there is a finite open cover of $(M,d)$ by sets of diameter at most $\e$, with Lebesgue number at least $\e/4$, and multiplicity at most $D$.
\end{lemma}

\begin{proof}
Let $x_1,\ldots,x_k$ be a maximal $\e/4$-separated set in $M$, so that 
$$M = \bigcup_{i=1}^k B(x_i,\e/4),$$
and let $\mathcal{U} = \{B(x_i,\e/2)\}_{i=1}^k$. Then $\mathcal{U}$ is an open cover of $M$ by sets of diameter at most $\e$, with Lebesgue number at least $\e/4$. Moreover, if 
$$x \in B(x_{i_1},\e/2) \cap \cdots \cap B(x_{i_m},\e/2),$$ 
then $\{x_{i_1},\ldots,x_{i_m} \}$ is an $\e/4$-separated set in the ball $B(x,\e/2)$, so the doubling property ensures that $m \leq D$.
\end{proof}

We will also need the following result, which is a special case of Proposition 5.8 in \cite{Sem96}. Namely, in the notation from that proposition, the statement we record is the case that $X=M=Z$ and the local contractibility function $\rho$ is linear.

\begin{lemma} \label{Sem}
Let $(M,d)$ be a compact metric space that is $D$-doubling, $L$-linearly locally contractible, has topological dimension $n$, and has $\diam(M) \geq 1$. There is $\delta >0$, depending only on $D$, $L$, and $n$, for which the following property holds. If $q_1,q_2 \colon M \rightarrow M$ are continuous maps with $d(q_1(x),q_2(x)) \leq \delta$ for all $x \in M$, then $q_1$ is homotopic to $q_2$ via a homotopy $h \colon [0,1] \times M \rightarrow M$ with $d(h_t(x),x) \leq 1/4$ for all $x \in M$ and $0\leq t \leq 1$.
\end{lemma}

We can now state and prove the main result of this section. In fact, its proof is quite similar to the proof of Lemma \ref{Sem}, though in some ways it is much simpler.

\begin{prop} \label{fg}
Let $(M,d)$ be a compact metric space that is $D$-doubling, $L$-linearly locally contractible, has topological dimension $n$, and has $\diam(M) \geq 1$. Then there is a finite simplicial complex $\mathcal{S}$ of dimension at most $D$, a $C$-Lipschitz map $f \colon M \rightarrow \mathcal{S}$, and a continuous map $g \colon \mathcal{S} \rightarrow M$ for which the following holds. If $p \colon \mathcal{S} \rightarrow \mathcal{S}$ is a continuous map with $p(\sigma) \subset \sigma$ for each sub-simplex $\sigma \subset \mathcal{S}$, then $g\circ p \circ f \colon M \rightarrow M$ is homotopic to $\id_M$ via a homotopy $h \colon [0,1] \times M \rightarrow M$ with $d(h_t(x),x) \leq 1/4$ for all $x \in M$ and $0\leq t \leq 1$. Here, the Lipschitz constant $C$ depends only on $D$, $L$, and $n$.
\end{prop}

\begin{proof}
Let $0<\delta<1$ be the constant from Lemma \ref{Sem}. Let $\mathcal{U} = \{U_i\}_{i=1}^\ell$ be a finite open cover of $M$ given by Lemma \ref{cover}, with $\e = \delta/((D!)(2L)^D)$. Let $\mathcal{S} = \Ner(\mathcal{U})$, which is a finite simplicial complex, and let $N$ be the maximal dimension of a simplex in $\mathcal{S}$. Then $N \leq D-1$, as the multiplicity of $\mathcal{U}$ is at most $D$. The Lebesgue number of $\mathcal{U}$ is at least $\e/4$, so the Lipschitz map $f \colon M \rightarrow \mathcal{S}$ constructed above has Lipschitz constant depending only on $D$ and $\e$, so only on $D$, $L$, and $n$.

Now let us construct $g$ by induction on skeleta. Choose $x_i \in U_i$, and let $g^{(0)} \colon \mathcal{S}^{(0)} \rightarrow M$ by $g^{(0)}(e_i) = x_i$. To define $g^{(1)} \colon \mathcal{S}^{(1)} \rightarrow M$, we proceed as follows. For $\sigma \subset \mathcal{S}$ a $1$-dimensional simplex (i.e., an edge), we have $\partial \sigma = \{e_i,e_j\}$ for some $1\leq i,j \leq \ell$. As $U_i \cap U_j \neq \emptyset$, there is $x \in U_i \cap U_j$ with $x_i,x_j \in B(x,\e)$. Applying $L$-linear local contractibility to this ball $B(x,\e)$, we find that there is a path from $x_i$ to $x_j$ of diameter at most $2L\e$. Define $g^{(1)}$ on $\sigma$ to be a parameterization of this path so that $g^{(1)}$ agrees with $g^{(0)}$ on $\partial \sigma = \{e_i,e_j\}$. Doing this for all edges gives a continuous map $g^{(1)} \colon \mathcal{S}^{(1)} \rightarrow M$ for which the image of each edge has diameter at most $2L\e$.

Suppose that we have constructed a continuous map $g^{(k)} \colon \mathcal{S}^{(k)} \rightarrow M$ for some $1\leq k < N$, such that the image of each $k$-dimensional simplex has diameter at most $(k!)(2L)^k\e$. We define $g^{(k+1)}$ as follows. Let $\sigma \subset \mathcal{S}$ be a $(k+1)$-dimensional simplex, so $\partial \sigma$ is contained in $\mathcal{S}^{(k)}$. In fact, $\partial \sigma$ is a union of $k+1$ simplices of dimension $k$, and the image of each under $g^{(k)}$ has diameter at most $(k!)(2L)^k\e$. As $g^{(k)}(\partial \sigma)$ is connected, we have 
$$\diam(g^{(k)}(\partial \sigma)) \leq ((k+1)!)(2L)^k\e \leq \diam(M)/L.$$
Choosing a point $x \in g^{(k)}(\partial \sigma)$, by linear local contractibility we can contract $g^{(k)}(\partial \sigma)$ inside the ball $B(x,L((k+1)!)(2L)^k\e)$. It is not difficult to obtain from this contraction a continuous map 
$$g^{(k+1)} \colon \sigma \rightarrow B(x,L((k+1)!)(2L)^k\e)$$ 
that agrees with $g^{(k)}$ on $\partial \sigma$ and has image equal to the trace set of the contraction of $g^{(k)}(\partial \sigma)$. Doing this for each $(k+1)$-dimensional simplex $\sigma$ in $\mathcal{S}$, we obtain a continuous map $g^{(k+1)} \colon \mathcal{S}^{(k+1)} \rightarrow M$ that extends $g^{(k)}$, such that the image of each $(k+1)$-dimensional simplex has diameter at most $((k+1)!)(2L)^{k+1}\e$.

As $\mathcal{S}^{(N)} = \mathcal{S}$, we obtain $g=g^{(N)} \colon \mathcal{S} \rightarrow M$ for which the image of any simplex in $\mathcal{S}$ has diameter at most $(N!)(2L)^N\e$. We claim that this is the desired map $g$ in the statement of the lemma.

To show this, let us fix a continuous map $p \colon \mathcal{S} \rightarrow \mathcal{S}$ with $p(\sigma) \subset \sigma$ for each simplex $\sigma \subset \mathcal{S}$. One could think of $p=\id_{\mathcal{S}}$ without loss of ideas. We claim that $g\circ p \circ f(x)$ is close to $x$ for each $x \in M$. Indeed, let $U_{i_1},\ldots,U_{i_m}$ be the sets in $\mathcal{U}$ that contain $x$, so $f(x) \in \sigma = \conv(e_{i_1},\ldots,e_{i_m})$. By assumption, $p(\sigma) \subset \sigma$, and by construction of $g$, we know that 
$$\diam(g(\sigma)) \leq (N!)(2L)^N \e \leq \delta/2.$$ 
As $e_{i_1} \in \sigma$, we have $x_{i_1} = g(e_{i_1}) \in g(\sigma)$, so we can bound
$$\begin{aligned}
d(x,g\circ p \circ f(x)) &\leq d(x,x_{i_1}) + d(x_{i_1}, g\circ p \circ f(x)) \\
&\leq \diam(U_{i_1}) + \diam(g(\sigma)) \\
&\leq \delta/2 + \delta/2 = \delta.
\end{aligned}$$
Applying Lemma \ref{Sem} to the maps $\id_M$ and $g \circ p \circ f$ gives the desired conclusion.
\end{proof}

\section{Codimension-1 volume bounds} \label{secend}

We are now ready to prove Theorem \ref{sep}, which we restate for convenience. Recall, from Remark \ref{isormk}, that Theorem \ref{isothm} follows immediately from it.

\begin{thmsep}
Let $(M,d)$ be a closed, connected, metric manifold of dimension $n \geq 1$ that is $D$-doubling and $L$-linearly locally contractible. If a closed set $S \subset M$ separates two points $x,y$ in $M$ with $\dist(S, \{x,y\}) \geq r$, then 
$$\mathcal{H}_{n-1}(S) \geq c \cdot r^{n-1},$$ 
where $c>0$ depends only on $n$, $D$, and $L$. In particular, for any closed set $S \subset M$, we have $\mathcal{H}_{n-1}(S) \geq c \cdot \seprad(S)^{n-1}$.
\end{thmsep}

\begin{proof}
The first statement directly implies the second, so we focus on the former. We can assume that $n \geq 2$, as the $n=1$ case follows from the trivial fact that sets of positive separation radius must be non-empty. Moreover, we may assume that $r=1$. Indeed, the assumptions of $D$-doubling and $L$-linear local contractibility are scale-invariant, and the desired conclusion scales appropriately. In particular, we then have $\diam(M) \geq 1$.

Let $\mathcal{S}$ be the simplicial complex from Proposition \ref{fg}, so that the dimension of each simplex in $\mathcal{S}$ is at most $D$. Let $f \colon M \rightarrow \mathcal{S}$ and $g \colon \mathcal{S} \rightarrow M$ be the corresponding maps, so that $f$ is $C$-Lipschitz, where $C$ also depends only on $D$, $L$, and $n$. Now, let $c_{n-1,D} >0$ be the constant given by Proposition \ref{proj}.

Suppose, for a contradiction, that $\mathcal{H}_{n-1}(S) \leq c_{n-1,D}/C^{n-1}$. Then $f(S) \subset \mathcal{S}$ has $\mathcal{H}_{n-1}(f(S)) \leq c_{n-1,D}$, so by Proposition \ref{proj}, there is a continuous map $p \colon \mathcal{S} \rightarrow \mathcal{S}$ with $p(\sigma) \subset \sigma$ for each sub-simplex $\sigma \subset \mathcal{S}$, and $p(f(S)) \subset \mathcal{S}^{(n-2)}$. Proposition \ref{fg} guarantees that there is a homotopy between $\id_M$ and $g \circ p \circ f$ that moves points in $M$ by distance at most $1/4$.

Restricting this homotopy to the set $S$, we see that $g \circ p \circ f|_S \colon S \rightarrow M$ is a continuous map that is homotopic to the inclusion $S \hookrightarrow M$ through maps whose images are disjoint from $\{x,y\}$. Let $S' = g \circ p \circ f(S)$, so by Lemma \ref{hom}, the induced homomorphism 
$$(g\circ p \circ f)^\ast \colon \check{H}^{n-1}(S') \rightarrow \check{H}^{n-1}(S)$$ 
is non-trivial. Notice, however, that this homomorphism factors into $(g \circ p \circ f)^\ast = (p \circ f)^\ast \circ g^\ast$, where 
$$g^\ast \colon  \check{H}^{n-1}(S') \rightarrow  \check{H}^{n-1}(p\circ f(S))$$ 
and 
$$(p \circ f)^\ast \colon \check{H}^{n-1}(p\circ f(S)) \rightarrow \check{H}^{n-1}(S).$$
As $p\circ f(S) \subset \mathcal{S}^{(n-2)}$, we have $\check{H}^{n-1}(p\circ f(S)) = 0$. This gives immediately that $(g\circ p \circ f)^\ast$ is trivial, which is a contradiction.

Thus, we conclude that $\mathcal{H}_{n-1}(S) > c_{n-1,D}/C^{n-1}$. This proves the theorem with constant $c = c_{n-1,D}/C^{n-1} >0$.
\end{proof}

\end{document}